\documentclass[article,a4paper,oneside,11pt]{memoir}

\counterwithout{section}{chapter}

\usepackage[T1]{fontenc}
\usepackage[utf8]{inputenc}
\usepackage{amsmath,amssymb,amsthm}
\usepackage{enumerate}

\newcommand{\nin}{\not\in}
\newcommand{\pr}{\mathbf P}
\newcommand{\e}{\mathbf E}
\newcommand{\Ima}{\mathrm{Im}\,}
\newcommand{\Coim}{\overline{\mathrm{Im}}\,}

\usepackage[british]{babel}

\title{Packing spanning graphs from separable families}

\author{Asaf Ferber
\thanks{Department of Mathematics, Yale University, New Haven, CT 06520, and
Department of Mathematics, MIT, Cambridge, MA 02139-4307. Emails:
asaf.ferber@yale.edu, and ferbera@mit.edu.}
\and
Choongbum Lee
\thanks{Department of Mathematics,
MIT, Cambridge, MA 02139-4307. Email: cb\_lee@math.mit.edu.
Research supported by NSF Grant DMS-1362326.}
\and
Frank Mousset
\thanks{Department of Computer Science, ETH Zürich, 8092 Zürich, Switzerland.
Email: frank.mousset@inf.ethz.ch.
Supported by grant no. 6910960 of the Fonds National de la Recherche, Luxembourg.}}

\newtheorem{theorem}{Theorem}[section]
\newtheorem{lemma}[theorem]{Lemma}

\newtheorem{corollary}[theorem]{Corollary}
\newtheorem{definition}[theorem]{Definition}
\newtheorem{claim}[theorem]{Claim}
\newtheorem{conjecture}{Conjecture}

\begin{document}

\maketitle

\begin{abstract}
  Let $\mathcal G$ be a separable family of graphs. Then for all positive
  constants $\epsilon$ and $\Delta$ and for every sufficiently large integer
  $n$, every sequence $G_1,\dotsc,G_t\in\mathcal G$ of graphs of order $n$ and
  maximum degree at most $\Delta$ such that $e(G_1)+\dotsb+e(G_t) \leq
  (1-\epsilon)\binom{n}{2}$ packs into $K_n$. This improves results of
  Böttcher, Hladk\'y, Piguet, and Taraz when $\mathcal G$ is the class of
  trees and of Messuti, Rödl, and Schacht in the case of a general
  separable family. The result also implies approximate versions of the
  Oberwolfach problem and of the
  Tree Packing Conjecture of Gyárfás and Lehel (1976) for the case that all
  trees have maximum degree at most $\Delta$.

  The proof uses the local resilience of random graphs and a special
  multi-stage packing procedure.
\end{abstract}

\section{Introduction}

A \emph{packing} of graphs $G_1,\ldots, G_t$ into a graph $G$ is a collection
of graph embeddings $f_i\colon G_i\to G$ whose images
$f_1(G_1),\dotsc,f_t(G_t)$ are pairwise edge-disjoint subgraphs of $G$. If such
a packing exists, then the graphs $G_1,\dotsc,G_t$ are said to \emph{pack} into
$G$. In this context, the graph $G$ is often called the \emph{target graph} or
the \emph{host graph}. The packing is said to be \emph{perfect} if it uses all
edges of the host graph, that is, if $e(G_1)+\dotsb + e(G_t) = e(G)$.

The problem of finding a perfect packing of identical copies of a graph $H$ in
$K_n$ has a long and rich history. If $H= K_k$, then perfect packings of $H$
into $K_n$ are also called \emph{block designs} or \emph{Steiner systems} and are a
central object of study in the field of design theory. One easily sees that a
perfect packing of $K_k$ into $K_n$ can only exist if $(k-1)| (n-1)$ and
$k(k-1)|n(n-1)$. In 1847, Kirkman proved that when $k=3$, then these necessary
conditions are also sufficient for the existence of a packing. The conjecture
that the conditions are actually sufficient for fixed $k$ and for sufficiently
large $n$ is known as the Existence Conjecture for block designs. The 1975
proof of this conjecture by Wilson~\cite{wilson1} is considered a breakthrough
result in combinatorics. If we move to more general graphs $H$, then once again
we see that a perfect packing of identical copies of $H$ into $K_n$ cannot
exist unless $e(H)|\binom{n}{2}$ and $\gcd(H) | (n-1)$, where the $\gcd$ of a
graph is the greatest common divisor of the vertex degrees. It is not too hard
to come up with examples where these conditions are satisfied, but where there
does not exist a perfect packing of $H$ into $K_n$; for example one cannot pack
two copies of $K_{1,3}$ into $K_4$. However, a remarkable theorem of Wilson
states that for every $H$ there is some $n_0(H)$ such that the divisibility
conditions \emph{are} sufficient whenever $n>n_0(H)$~\cite{wilson2}. The number
$n_0(H)$ is quite large and at present there is little hope of determining what
its smallest possible value might be.

This leaves open the question of packing graphs where the number of vertices
can be close to $n$. As mentioned above, the obvious divisibility conditions
are not always sufficient to guarantee a perfect packing in this case. Clearly, not
even two graphs need to pack in general: for example two stars on $2n$ vertices
do not pack into $K_{2n}$. If we wish to find perfect (or nearly perfect) packings of
general graphs, then we need to find some way to exclude these and similar
cases. One such way comes in the form of the following beautiful 1976
conjecture of Gy\'arfas and Lehel~\cite{gyarfas-lehel}.

\begin{conjecture}[Tree Packing Conjecture]\label{tpc}
  Any collection $T_1,\dotsc,T_n$ of trees such that $v(T_i) = i$ packs into
  $K_n$.
\end{conjecture}

Note that since each $T_i$ has exactly $i-1$ edges, the packing of the trees in
question would necessarily be perfect. Although Conjecture~\ref{tpc} is still
open, there are several partial results. The conjecture is known to be true
when the trees satisfy various structural requirements,
see~\cite{dobson1,dobson2,dobson3,gyarfas-lehel,roditty}. For arbitrary trees,
Bollob\'as~\cite{bollobas} proved that one can at least pack the small trees
$T_1,\dotsc,T_{n/\sqrt{2}}$ into $K_n$, which can be improved to $\sqrt{3}n/2$
if the Erd\H{o}s-S\'os conjecture is true. The situation for the larger trees
is much worse: it is only known that the five trees
$T_{n-4},T_{n-3},T_{n-2},T_{n-1},T_n$ pack into $K_n$~\cite{zak} (the previously best known
result for three trees $T_{n-2},T_{n-1},T_n$ was proved in~\cite{hbk}), although Balogh and
Palmer~\cite{balogh-palmer} proved
that one can pack the $n^{1/4}/10$ largest trees into $K_{n+1}$ (note that an
extra vertex is needed), and they also proved that one can pack the $n^{1/3}/4$
largest trees into $K_n$ if each has maximum degree at least $2n^{2/3}$. In any
case, packing even a small linear fraction of the large trees appears to be a
very hard problem.

The problem becomes more tractable if one additionally assumes that each tree
$T_i$ has bounded maximum degree. A well-known result of Böttcher, Hladk\'y,
Piguet, and Taraz~\cite{bhpt} shows that in this setting, one can achieve
nearly perfect packings of trees, provided that the size of the trees is
bounded away from $n$:

\begin{theorem}[Böttcher, Hladk\'y, Piguet, and Taraz~\cite{bhpt}]\label{thm:bhpt}
  For all positive constants $\epsilon$ and $\Delta$ and for every sufficiently
  large integer $n$, the following holds.
  Every collection $T_1,\dotsc,T_t$ of trees of order at most $(1-\epsilon)n$ and maximum
  degree at most $\Delta$ such that $e(T_1)+\dotsb+e(T_t)\leq (1-\epsilon)\binom{n}{2}$
  packs into $K_n$.
\end{theorem}

In particular, their theorem yields the following asymptotic version of the
Tree Packing Conjecture.

\begin{corollary}\label{cor1}
  For all positive constants $\epsilon$ and $\Delta$ and for every sufficiently
  large integer $n$, every sequence $T_1,\dotsc,T_{(1-\epsilon)n}$ of trees
  with $v(T_i) = i$ and maximum degree at most $\Delta$ packs into $K_n$.
\end{corollary}

The fact that the graphs $T_i$ from Conjecture~\ref{tpc} are assumed to be
trees is very important (otherwise the conjecture is simply not true). However,
there is no reason why Theorem~\ref{thm:bhpt} should hold only for trees. In
fact, we believe that Theorem~\ref{thm:bhpt} remains true if one replaces the
word ``trees'' by ``graphs''. While this is not known, a result similar to
Theorem~\ref{thm:bhpt} holds at least for all \emph{separable families} of
graphs. In the literature, a family $\mathcal G$ of graphs is typically said to
be separable if it satisfies a \emph{separator theorem}, meaning there exist
constants $c,\alpha<1$ such that the removal of $O(v(G)^\alpha)$ vertices from
any $G\in \mathcal G$ can partition the graph into disjoint subgraphs of size
at most $c\cdot v(G)$. However, following~\cite{mrs}, we adopt a slightly more
general definition:

\begin{definition}[separable]\label{def:separable}
  A graph $G$ is \emph{$(\delta,K)$-separable} if by removing $\delta v(G)$
  vertices from $G$ one can obtain a graph in which every component has size at
  most $K$. A family $\mathcal G$ of graphs is \emph{separable} if for every
  $\delta>0$ there is some $K$ such that each graph $G\in \mathcal G$ is
  $(\delta,K)$-separable.
\end{definition}

As shown in~\cite{mrs}, every family of graphs that is separable in the
familiar sense given above is also separable according to
Definition~\ref{def:separable}. It is worth mentioning that by the famous
separator theorem of Alon, Seymour and Thomas~\cite{ast}, all non-trivial
minor-closed families of graphs (such as planar graphs, graphs of bounded
treewidth, etc.) are separable. In particular, the class of trees is separable.
Messuti, R\"odl, and Schacht proved the following generalisation of
Theorem~\ref{thm:bhpt} to separable families of graphs~\cite{mrs}.

\begin{theorem}[Messuti, Rödl, and Schacht~\cite{mrs}]\label{thm:mrs}
  Assume that $\mathcal G$ is a separable family of graphs. Then for all
  positive constants $\epsilon$ and $\Delta$ and for every sufficiently large
  integer $n$, the following holds. Every collection $G_1,\dotsc,G_t\in
  \mathcal G$ of $t\leq (1-\epsilon)n$ graphs of order at most $(1-\epsilon)n$ and maximum degree at
  most $\Delta$ such that $e(G_1)+\dotsb+e(G_t)\leq
  (1-\epsilon)\binom{n}{2}$ packs into $K_{n}$.
\end{theorem}

Note that in Theorems~\ref{thm:bhpt} and \ref{thm:mrs}, the graphs being packed
have a number of vertices which is bounded away from $n$. This situation is not
untypical: in graph embedding problems, it is almost always easier to embed a
nearly spanning structure than it is to embed a spanning structure. However, in
light of the fact that the hard part of Conjecture~\ref{tpc} is the embedding
of the very large trees, it is highly desirable to develop methods that allow
the packing of \emph{spanning} subgraphs. The contribution of this paper is the
following stronger version of Theorem~\ref{thm:mrs}:

\begin{theorem}\label{thm:main}
  Assume that $\mathcal G$ is a separable family of graphs. Then for all
  positive constants $\epsilon$ and $\Delta$, there exists $n_0 =
  n_0(\epsilon,\Delta,\mathcal G)$ such that the following holds for all $n\geq
  n_0$. Every collection $G_1,\dotsc,G_t\in \mathcal G$ of graphs of order at
  most $n$ and maximum degree at most $\Delta$ such that
  $e(G_1)+\dotsb+e(G_t)\leq (1-\epsilon)\binom{n}{2}$ packs into $K_{n}$.
\end{theorem}

While the theorem applies to all separable families, the result is new even
when $\mathcal G$ is the family of trees. From it, one also immediately obtains
a stronger approximate version of the Tree Packing Conjecture:

\begin{corollary}
  For all positive constants $\epsilon$ and $\Delta$ and for every sufficiently
  large integer $n$, every collection $T_{\epsilon n},\dotsc,T_n$ of trees
  with $v(T_i) = i$ and maximum degree at most $\Delta$ packs into $K_n$.
\end{corollary}

Since the family of all $2$-regular graphs is separable, Theorem~\ref{thm:main}
also yields an approximate version of the `Oberwolfach problem' (proposed by
Ringel in 1967) which states that given any $2$-regular graph $F$ on $n$
vertices, where $n$ is odd, there exists a perfect packing of copies of $F$
into $K_n$. Bryant and Scharaschkin~\cite{bryant-scharaschkin} proved that the
exact statement of the Oberwolfach problem is correct for infinitely many
values of $n$.

There are many possible ways in which Theorem~\ref{thm:main} could be
strengthened. Most obviously, the restriction that the graphs are from a
separable family is likely unnecessary; the theorem as stated should hold for
any bounded-degree graphs $G_1,\dotsc,G_t$. Secondly, one would like to know
the dependence of $n_0$ on the parameters $\epsilon$ and $\Delta$. It is not
clear what kind of dependence follows from our proof, but it is certainly very
far from optimal. Thirdly, it would be interesting to have an efficient
algorithm which, given the graphs $G_1,\dotsc,G_t$, produces a packing into
$K_n$. Our proof does not give such an algorithm, although we pretend
(without proof) that if $\mathcal G$ is the family of trees, then a variation
of our argument yields a randomized algorithm which packs the given trees with
high probability and which runs in time $O_{\epsilon,\Delta}(n^c)$ for some
reasonable absolute constant $c>0$. Lastly, there is some interest in the problem of
packing graphs of bounded maximum degree into random and pseudorandom graphs.
Our proof of Theorem~\ref{thm:main} can be modified with little effort to show
that for any constant $p\in (0,1]$, a.a.s.\ the given graphs $G_1,\dotsc,G_t$
pack into the random graph $G(n,p)$ (see below for the definitions), provided
that $e(G_1)+\dotsb+e(G_t) \leq (1-\epsilon)p\binom{n}{2}$. In its current
form, our proof does not provide packings into any model of pseudorandom
graphs.

Very recently, Kim, Kühn, Osthus, and Tyomkyn~\cite{kkot} have managed
to address the first and last points mentioned above: the condition of
separability in Theorem~\ref{thm:main} is not necessary -- and in fact the
graphs do not just pack into $K_n$ but into any dense quasi-random graph.

Before continuing, we give some further pointers to the literature. Recently,
Keevash extended the results of Wilson on perfect packings of $K_k$ to
quasi-random hypergraphs~\cite{keevash}. Barber, Kühn, Lo, and Osthus studied
perfect packings of triangles or even cycles into graphs of large minimum
degree~\cite{bklo}. Packings of arbitrary small graphs into graphs of large
minimum degree were considered by Glock, Kühn, Lo, Montgomery, and Osthus
in~\cite{gklmo}. A well-known conjecture of Kelly states that every regular
tournament on $n$ vertices can be decomposed into $(n-1)/2$ directed Hamilton
cycles. The proof of a strong form of this conjecture for large $n$ was
given by Kühn and Osthus~\cite{ko1}; the method used seems to be widely
applicable, see~\cite{ko2}.

\subsection*{Notation}
We use $[n]$ for the set $\{1,2,\dotsc,n\}$. If $G$ is a graph and $X$ is
a subset of the vertices, then $N_G(X)$ denotes the neighbourhood of $X$ in $G$, i.e.,
the set of all vertices adjacent to a vertex of $X$.
We define $G-X$ as the subgraph of $G$ induced on $V(G) \setminus X$.

The disjoint union of two graphs $G$ and $H$ is written as $G+H$. If $G$ is a
graph and $n$ is an integer, then we write $G\times n$ for the graph consisting
of $n$ disjoint copies of $G$.

If $f\colon G\to H$ is a graph embedding, then $\Ima f$ denotes
the image of $f$ (considered as a subset of $V(H)$)
and $\Coim f$ denotes the complement $V(H)\setminus \Ima f$ of the image.
On the other hand, we write $f(G)$ to denote the \emph{subgraph} of $H$ that
$G$ is mapped to (which is always isomorphic to $G$).

We use two closely related random graph models. $G(n,p)$ denotes the random
graph on $n$ vertices in which every edge is present independently with
probability $p$. $B(n,p)$ denotes the random \emph{bipartite} graph with parts
of size $n$ in which every edge that goes between the parts is present
independently with probability $p$.

A sequence of events $A_n$ is said to occur asymptotically almost surely
(a.a.s.) if $\pr[A_n]\to1$ as $n\to\infty$. By a common abuse of language, we
simply say that the event $A_n$ occurs a.a.s..

\section{Preliminaries}

For the proof we need some probabilistic tools. First of all, we use the
well-known Chernoff bounds~\cite[Theorem 2.1]{mcdiarmid}.

\begin{lemma}[Chernoff]\label{lemma:chernoff}
  Assume that $X_1,\dotsc, X_n$ are jointly independent binary random variables such
  that $\pr[X_i=1] = p$ and $\pr[X_i=0] = 1-p$. Let $X = X_1+\dotsb+X_n$. Then
  for any $\epsilon \in (0,1)$
  \[ \pr[|X - \e[X]| \geq \epsilon np] \leq 2e^{-\epsilon^2np/3}. \]
\end{lemma}

Additionally, we use the following variant of the Azuma-Hoeffding inequality.
The formulation below follows by applying the usual Azuma-Hoeffding
inequality~\cite[Theorem 27]{fan-lu} to the supermartingale $Y_1,\dotsc,Y_n$
defined by $Y_k := \sum_{i=1}^k (X_i-\mu_i)$.

\begin{lemma}[Azuma-Hoeffding]\label{lemma:azuma}
  Let $X_1,\dotsc,X_n$ be random variables where $0\leq X_i \leq c_i$ for all $i$
  and where there are real numbers $\mu_1,\dotsc,\mu_n$ such that
  \[ \e[X_i\mid X_{i-1},\dotsc,X_1] \leq \mu_i \]
  holds for all $i$. Let $X = X_1+ \dotsb + X_n$ and $\mu =
  \mu_1+\dotsb+\mu_n$. Then for any $\epsilon >0$
  \[ \pr[X \geq (1+\epsilon)\mu] \leq e^{-\epsilon^2\mu^2/(2\sum_{i=1}^nc_i^2)}. \]
\end{lemma}

In our proof, the following notion will be very important:

\begin{definition}[Local resilience]
  Let $\mathcal P$ be any monotone increasing graph property. The \emph{local
  resilience} of a graph $G$ with respect to $\mathcal P$ is the minimum number
  $r$ such that by deleting at most $r$ edges at each vertex $v$ of $G$, one
  can obtain a graph not in $\mathcal P$.
\end{definition}

We will only need to consider the case where $\mathcal P$ is the
property of containing a perfect matching. For this, we have the following
lemma, whose proof we omit (it is almost identical to the proof
of Theorem 2.3 in the seminal paper of Sudakov
and Vu~\cite{sudakov-vu-resilience}). We remark that although the lemma gives a
local resilience of almost $1/2$, for our purposes, any positive constant would
suffice.

\begin{lemma}\label{lemma:matchings-resilience}
  The following holds for every constant $\epsilon>0$ and every function $p =
  p(n) = \omega(\log n/n)$. With probability $1-o(n^{-1})$ the local
  resilience of the random bipartite graph $B(n,p)$ with respect to the property of
  containing a perfect matching is greater than $(1-\epsilon)np/2$.
\end{lemma}

The following lemma (a special case of a result of Bárány and Doerr~\cite{barany-doerr}) will
turn out to be surprisingly useful.

\begin{lemma}[Bárány and Doerr]\label{lemma:averaging}
  Let $d$ be a positive integer.
  For every finite set $A\subseteq [0,1]^d$ and every positive integer
  $m$, there is a partition of $A$ into $m$ parts $A_1,\dotsc,A_m$ such
  that for all $i\in [m]$
  \[ \sum_{a\in A_i} a \in \frac{1}{m}\sum_{a\in A}a + 3d [-1,1]^d. \]
\end{lemma}

Finally, we shall need a result on packing nearly spanning graphs with small
components into random graphs. For brevity we will adopt the following definition.

\begin{definition}
  Let $\mathcal{BC} (n,K)$ denote the family of all graphs of order at most $n$
  in which every component has size at most $K$.
\end{definition}

The next lemma sates that graphs from $\mathcal{BC}((1-\epsilon)n,K)$ pack
asymptotically optimally into dense random graphs on $n$ vertices, and
moreover, this packing can be chosen to be random-like in the sense that fixed
subsets of the graphs are spread evenly over the vertices of $G(n,p)$. Since the proof
is rather involved, we postpone it to Section~\ref{sec:packing-small}. As a
side remark, this lemma is the only part of our proof that is not algorithmic
(in the sense that its proof does not suggest an efficient randomized algorithm
that produces the desired packing with high probability).
However, if all the components are trees, then one can use the method of~\cite{bfkl}
to obtain an algorithmic proof.

\begin{lemma}\label{lemma:bounded-gnp}
  Let $\epsilon,\alpha,\beta,K>0$ and $p\in (0,1]$ be constants and let $G\sim G(n,p)$.
  Assume that $G_1,\dotsc,G_t$ are graphs in $\mathcal{BC}((1-\epsilon)n,K)$ such that
  \[ e(G_1)+\dotsb+e(G_t)\leq (1-\epsilon)\binom{n}{2}p,\]
  where each graph $G_i$ has at least $(1-2\epsilon)n$ vertices. Moreover,
  assume that $A_1,\dotsc,A_t$ and $B_1,\dotsc,B_t$ are sets where for all $i$
  we have $A_i,B_i\subseteq V(G_i)$, $|A_i|\leq \alpha n$, and $|B_i|\leq \beta
  n$. Then a.a.s.\ there exists a packing $\{f_i\colon G_i\to G\}$ of the graphs
  $G_1,\dotsc,G_t$ into $G$ with the following properties:
  \begin{enumerate}[(1)]
    \item for every vertex $v$ of $G$ there are at most $2\alpha t$ values $i\in [t]$ such that
      $v\in f_i(A_i)$ and at most $4(\beta+\epsilon) t$ values $i\in [t]$ such that
      $v\in f_i(B_i) \cup \Coim f_i$;
    \item for any two distinct vertices $u,v$ of $G$ there are at most $5(\beta
      + \epsilon)^2t$ values $i\in [t]$ such that $\{u,v\}\subseteq
      f_i(B_i)\cup \Coim f_i$.
  \end{enumerate}
\end{lemma}

\section{Proof of Theorem~\ref{thm:main}}

Let $\mathcal G$ be a separable family of graphs and let positive constants
$\epsilon$ and $\Delta$ be given. Additionally, fix constants
$\gamma,\delta,\zeta,p_0,K$ such that
\[ 0 < K^{-1} \ll \delta\ll \zeta \ll \gamma \ll p_0 \ll \epsilon, \]
where the notation $a\ll b$ means that $a$ should be sufficiently small
depending on $b$. We assume throughout that $n$ is a sufficiently large
integer.

Let $G_1,\dotsc,G_t\in \mathcal G$ be graphs of order at most $n$ and maximum
degree at most $\Delta$ such that $e(G_1)+\dotsb + e(G_t) \leq
(1-\epsilon)\binom{n}{2}$. The goal is to pack these graphs into $K_n$.
Note that by combining two graphs into one if they have fewer than $n/2$
non-isolated vertices, we can assume without loss of generality that at most
one graph $G_i$ has less than $n/4$ edges, which implies that $t\leq 2n$. By
adding isolated vertices, we can also assume that each graph $G_i$ has exactly $n$
vertices.

Since $\mathcal G$ is a separable family, each graph $G_i$ contains a set
$S_i\subseteq V(G_i)$ of size $\delta n$ such that $G_i-S_i$ only has
components of size at most $K$ (note that we will consistently omit rounding
brackets throughout the proof). For each graph $G_i$ we now fix such a set $S_i$.

An independent set $I$ in a graph $G$ is said to be \emph{2-independent} if any two
vertices of $I$ have disjoint neighbourhoods in $G$. Since each $G_i$ has maximum degree
at most $\Delta$, we can greedily find a $2$-independent set $I_i$ of size $\gamma n$
in $G_i$ that is disjoint from $S_i$.

The general idea of the proof is to proceed in three phases: in the first phase
we pack all the graphs $G_i-S_i-I_i$, in the second phase we extend this packing
to a packing of the graphs $G_i-I_i$, and in the final phase, we extend that packing
to a packing of the graphs $G_i$.

\subsection{Preparations}

Let $M:= \gamma^{-2}$ and fix a collection $Z^{(1)},\dotsc,Z^{(M)}$ of $M$
disjoint subsets of $V(K_n)$ containing $\zeta n$ vertices each (this is
possible since $M\zeta n \leq n$).

Now let $p := \gamma^2(1-p_0)$ and split $K_n$ into $M+1$ edge-disjoint random
subgraphs $\Gamma^{(0)},\dotsc,\Gamma^{(M)}$ as follows. First, let $\{X_e\mid e\in
E(K_n)\}$ be jointly independent random variables distributed uniformly on the
interval $[0,1$]. Then we define $\Gamma^{(0)}$ to be the subgraph of $K_n$ with the
edge set $\{e \mid X_e \in [0,p_0]\}$, and for $k\in [M]$ we define $\Gamma^{(k)}$ to be the
subgraph with edge set $\{e\mid X_e\in (p_0 +(k-1)p, p_0 + kp]\}$.

The graphs $\Gamma^{(0)},\Gamma^{(1)},\dotsc,\Gamma^{(M)}$ partition $K_n$ into
edge-disjoint subgraphs. Moreover, one can check that $\Gamma^{(0)}$ is distributed as
$G(n,p_0)$ and that the graphs $\Gamma^{(1)},\dotsc, \Gamma^{(M)}$ are distributed as
$G(n,p)$.

Similarly, we will split our collection of graphs $G_1,\dotsc,G_t$ into $M$ groups as follows.
For each $i\in [t]$ define the vector $a_i := (1,\frac{2e(G_i)}{\Delta n})\in
[0,1]^2$, and let $A:=\{a_1,\ldots,a_t\}$. By applying
Lemma~\ref{lemma:averaging} to $A$
we obtain a partition
\[A:=\bigcup_{1\leq k\leq M} A_k\] for which
\[\sum_{a\in A_k}a\in \frac{1}{M}\sum_{a\in A}a \pm 6[-1,1]^2.\] Clearly, since
the first coordinate of each $a_i$ is $1$, it follows that $|A_k|\leq \frac
tM+6$ for all $k$. Moreover, since the second coordinate of each $a_i$ is
$\frac{2e(G_i)}{\Delta n}$, it follows that
\[ \sum_{i: a_i\in A_k}\frac{e(G_i)}{\Delta n}\in \frac{\sum_ie(G_i)}{M\Delta n}\pm 6. \]
By setting $\mathcal B^{(k)}:=\{i\in [t]\mid  a_i\in A_k\}$ we obtain a partition
\[ \{1,\dotsc,t\} = \bigcup_{1\leq k \leq M} \mathcal B^{(k)} \]
which satisfies $\max_k{|\mathcal B^{(k)}|} \leq \frac{t}{M}+6 \leq 3\gamma^2 n$ and
\begin{equation}\label{eq:edges}
  \sum_{i\in \mathcal B^{(k)}} e(G_i) \leq
  \frac{e(G_1) + \dotsb + e(G_t)}{M} + 6\Delta n \leq
  (1-\epsilon/2)\binom{n}{2}p,
\end{equation}
where we used that $e(G_1)+\dotsb+e(G_t) \leq (1-\epsilon)\binom{n}{2}$.

\subsection{Phase I: packing the $G_i-I_i-S_i$}

For each $k\in[M]$, we will pack the collection $\{G_i-I_i-S_i\mid i\in
\mathcal B^{(k)}\}$ into the graph $\Gamma^{(k)}-Z^{(k)}$. The edges of $\Gamma^{(k)}$
that are incident with $Z^{(k)}$ will only be used in the second phase. We also
note that this leaves the graph $\Gamma^{(0)}$ completely untouched; the edges
of $\Gamma^{(0)}$ are reserved for the final phase.

Our goal is to apply Lemma~\ref{lemma:bounded-gnp}. Let us first set everything up.
Since $Z^{(k)}$ was fixed beforehand, we can see that $\Gamma^{(k)}-Z^{(k)}$ is a random graph
distributed as $G((1-\zeta)n,p)$. As $\zeta \ll \gamma$ there is some $\xi>0$ so that
\[ 1-\delta-\gamma = (1-\xi)(1-\zeta). \]
In fact, as $\zeta,\delta \ll \gamma$, we have (say) $\gamma/2 \leq \xi \leq 2\gamma$.
By the definition of $S_i$ and using $v(G_i-I_i-S_i)= (1-\delta-\gamma)n$, we
then have $G_i-I_i-S_i\in \mathcal BC((1-\xi)(1-\zeta)n, K)$ and
$v(G_i-I_i-S_i) \geq (1-2\xi)(1-\zeta) n$.
Let $A_i$ be the set of neighbours of $S_i$ in $G_i-I_i-S_i$ and let $B_i$ be
the set of neighbours of $I_i$ in $G_i-I_i-S_i$. Note that $|A_i| \leq \Delta
|S_i| \leq \Delta\delta n$ and $|B_i| \leq \Delta|I_i|\leq \Delta \gamma n$.

Now we can apply Lemma~\ref{lemma:bounded-gnp} and obtain for each $k\in [M]$ a packing
\[ \{f_i\colon G_i-I_i-S_i\to \Gamma^{(k)}-Z^{(k)}\mid i \in \mathcal B^{(k)}\} \]
with the following properties (using $\max_k|\mathcal B^{(k)}|\leq 3\gamma^2 n$):
\begin{enumerate}[(P1)]
  \item for every vertex $v$ of $\Gamma^{(k)}-Z^{(k)}$ there are at most
    $6\Delta\delta \gamma^2n$ values $i\in\mathcal B^{(k)}$ such that $v\in f_i(A_i)$
    and at most $12(\Delta \gamma + \xi)\gamma^2 n \leq 50\Delta \gamma^3 n$
    values $i\in \mathcal B^{(k)}$ such that $v \in f_i(B_i)\cup
    (V(K_n)\setminus (Z^{(k)}\cup \Ima f_i))$;
  \item for any two distinct vertices $u,v,$ of $\Gamma^{(k)}-Z^{(k)}$ there are at most
    $15(\Delta \gamma +\xi)^2 \gamma^2n\leq 200 \Delta^2\gamma^4 n$ values $i\in
    \mathcal B^{(k)}$ such that $\{u,v\}\subseteq f_i(B_i) \cup (V(K_n)\setminus
    (Z^{(k)}\cup \Ima f_i))$.
\end{enumerate}

\subsection{Phase II: extending to a packing of the $G_i-I_i$}

The next step is to extend the embeddings $f_i\colon G_i-I_i-S_i\to
\Gamma^{(k)}-Z^{(k)}$ constructed above to edge-disjoint embeddings $g_i \colon
G_i-I_i\to \Gamma^{(k)}$. For this, we will use essentially the same greedy
procedure that was used in the proof of Theorem~\ref{thm:mrs}~\cite{mrs}. Fix
some $k\in [M]$. For each $i\in \mathcal B^{(k)}$, we will embed $S_i$ into
$Z^{(k)}$ using the edges of $\Gamma^{(k)}$.

First, observe that by the Chernoff bound (Lemma~\ref{lemma:chernoff}) and the
union bound, a.a.s.\ every set of at most $\Delta$ vertices of $\Gamma^{(k)}$ has
$(1+o(1))|Z^{(k)}|p^{\Delta} \geq \zeta \gamma^{2\Delta} n/2$ common neighbours in
$Z^{(k)}$. From now on, we assume that this is the case.

Now we process the values $i\in \mathcal B^{(k)}$ one by one. Fix some $i$ and
assume for an induction that we already have edge-disjoint embeddings
$\{g_j \colon G_j-I_j\to \Gamma^{(k)}\mid j\in \mathcal B^{(k)}\cap [i-1]\}$ with the
property that for each vertex $v$, there are at most $\zeta^2 n$
values $j\in \mathcal B^{(k)}\cap [i-1]$ such that $v\in g_j(S_j)$.
We show that we can also embed $S_i$ to obtain  a packing of
$\{G_j-I_j\mid j\in \mathcal B^{(k)}\cap [i]\}$ which does not violate this property.

Order the vertices of $S_i$ as $v_1,\dotsc,v_{\delta n}$. Note that the
neighbours of $S_i$ in $G_i-S_i-I_i$ are already embedded by $f_i$, and it only remains
to embed the vertices in $S_i$. We embed these vertices one by one, where we
need to embed each vertex into the common neighborhood of the images of the
neighbours that we have embedded up to that point. When doing this we never
want to use edges that have been used by an embedding of one of the graphs
$G_j-I_j$ where $j< i$.

Suppose that we have embedded $v_1,\dotsc,v_{\ell-1}$ and want to embed
$v_\ell$. Let $N_\ell$ be the set of the neighbours of $v_\ell$ that have been
embedded up to that point and let $\tilde N_\ell$ be the image of this set in $\Gamma^{(k)}$.
Since $N_\ell$ contains at most $\Delta$ vertices, we know that $\tilde N_\ell$
has at least $\zeta\gamma^{2\Delta} n/2$ common neighbours in $Z^{(k)}$ (by the
observation above). We need to embed $v_\ell$ into a common neighbour $u$ of
$\tilde N_\ell$ in $Z^{(k)}$ that satisfies the following three properties: (i) none of the
edges between $\tilde N_\ell$ and $u$ has been used by an embedding of a graph
$G_j-I_j$ where $j \in \mathcal B^{(k)}\cap [i-1]$; (ii) $u$ is not the image
of a vertex $v_1,\dotsc,v_{\ell-1}$; (iii) there are at most $\zeta^2n-1$
values $j\in \mathcal B^{(k)}\cap [i-1]$ such that $u\in f_j(S_j)$. If such a
common neighbour $u$ exists, then we can embed $v_\ell$ to $u$ and continue
with the next vertex.

We now show that such a vertex exists.
Firstly, by (P1) and by the induction hypothesis, and using $\zeta\ll \gamma$,
there exist at least $\zeta\gamma^{2\Delta} n/4$ common neighbours of $\tilde
N_\ell$ that satisfy (i). Moreover, since $\ell \leq \delta n \leq
\zeta\gamma^{2\Delta}/8$, among these there are at least
$\zeta\gamma^{2\Delta}n/8$ that also satisfy (ii). Finally, the total number of
vertices in $Z^{(k)}$ used by previous embeddings is at most $\delta n\cdot
|\mathcal B^{(k)}| \leq \delta n^2$ (counted with multiplicities), which is not
enough to cover each of those $\zeta\gamma^{2\Delta} n/8$ vertices $\zeta^2 n$
times (as $\delta \ll \zeta \ll \gamma$). So there must be a common neighbour
satisfying (i), (ii), and (iii), which can be the image of $v_\ell$.

This procedure results in packings
$\{g_i \colon G_i-I_i \to \Gamma^{(k)} \mid i\in \mathcal B^{(k)}\}$, one for each $k\in [M]$.
The union of these packings is a packing
\[ \{g_i \colon G_i-I_i\to K_n\setminus \Gamma^{(0)}\mid i\in [t]\}. \]
We finish this subsection by stating the crucial property that will allow us to
complete the packing of the graphs $G_i-I_i$ to a packing of the graphs $G_i$:

\begin{claim}\label{claim:balance}
  For every vertex $v$ of $K_n$, there are at most $\gamma^{0.9} n$ values $i\in [t]$
  such that \[ v\in g_i(N_{G_i}(I_i)) \cup \Coim g_i.\]
  For any two distinct vertices $u,v$ of $K_n$, there are at most $\gamma^{1.9} n$
  values $i\in [t]$ such that
  \[ \{u,v\} \subseteq g_i(N_{G_i}(I_i)) \cup \Coim g_i. \]
\end{claim}
\begin{proof}
  Let $F_i := g_i(N_{G_i}(I_i)) \cup \Coim g_i$.
  For a fixed $k\in [M]$ and any vertex $v\in V(K_n)\setminus Z^{(k)}$, by (P1)
  there are at most $50\Delta^3\gamma^3 n$ values $i\in \mathcal B^{(k)}$ such
  that $v \in F_i$. In total these cases contribute at most $50\Delta^3 \gamma n$ values of $i$.
  On the other hand, for any $v$, there is at most one $k\in
  [M]$ such that $v\in Z^{(k)}$ (as the $Z^{(k)}$ are disjoint sets), which contributes
  at most $\max_k |\mathcal B^{(k)}|\leq 3\gamma^2 n$ values of $i$. This gives the first statement.
  The second statement follows along the same lines, using (P2) instead of (P1).
\end{proof}

\subsection{Auxiliary graphs and local resilience}

Now that we have a packing $\{g_i \colon G_i-I_i\to K_n\setminus \Gamma^{(0)}\mid
i\in [t]\}$, we need to complete it into a packing $\{h_i \colon G_i\to
K_n \mid i\in [t]\}$. For this we will only use edges of $\Gamma^{(0)}$. However,
since the graphs are spanning, we cannot embed vertices one by one
as we did in Phase II. Luckily, the sets $I_i$ have a much nicer structure than
the sets $S_i$: they are $2$-independent sets. This will allow us
to use a method based on perfect matchings in bipartite
graphs to embed these sets.

For each $i\in [t]$, we define an auxiliary bipartite graph $B_i$ as follows:
\begin{itemize}
  \item the parts of $B_i$ are $\{ g_i(N_{G_i}(v))\mid v\in I_i\}$ and $\Coim g_i$
    (note that these are both sets of size $\gamma n$);
  \item there is an edge between $X$ and $x$ if $x$ lies in the common neighbourhood
    of $X$ in $\Gamma^{(0)}$.
\end{itemize}

Note that since $I_i$ is $2$-independent, the sets $g_i(N_{G_i}(v))$ are
disjoint for all $v\in I_i$. Observe also that it is possible to extend $g_i$
to an embedding $h_i\colon G_i\to K_n$ if and only if $B_i$ contains a perfect
matching: that is, if $x$ is matched to $X = g_i(N_{G_i}(v))$ then we can embed $v$ as
$x$. Since our goal is to extend the embeddings in an ``edge-disjoint way'', we
actually need to claim that $B_i$ is ``robust'' in some sense with respect to
the property of containing a perfect matching. For this we use the notion
of local resilience and show:

\begin{claim}\label{claim:resilience}
  The local resilience of $B_i$ with respect to containing a perfect matching
  is a.a.s.\ greater than $\gamma^{1.1} n$ for every $i\in [t]$.
\end{claim}

\begin{proof}
  Fix some $k\in [M]$.
  We start the proof with an observation about the joint distribution of the
  pair $(\Gamma^{(0)},\Gamma^{(k)})$. Recall that $\Gamma^{(k)}$ is
  the graph of edges $e\in E(K_n)$ for which the uniform random variable
  $X_e\in [0,1]$ takes a value in the interval $(p_0+(k-1)p,p_0+kp]$. Thus the
  marginal distribution of $\Gamma^{(k)}$ is $G(n,p)$. Now suppose we condition
  on the value of $\Gamma^{(k)}$, that is, we expose the set of edges $e$ for
  which $X_e\in (p_0+(k-1)p,p_0+kp]$. Then, by symmetry, the remaining values
  $X_e$ for $e\nin E(\Gamma^{(k)})$ are independent random variables
  distributed uniformly in the set $[0,p_0+(k-1)p] \cup (p_0+kp,1]$. By
  definition, $\Gamma^{(0)}$ contains exactly those edges $e\in E(K_n)$ for
  which $X_e\in [0,p_0]$. Thus each edge $e\nin E(\Gamma^{(k)})$ is an edge of
  $\Gamma^{(0)}$ independently and with probability $p_0/(1-p)$. In other
  words, the conditional distribution of $\Gamma^{(0)}$ given $\Gamma^{(k)}$ is
  $G(n,p_0/(1-p))\setminus \Gamma^{(k)}$.
  This shows that to generate the pair $(\Gamma^{(0)},\Gamma^{(k)})$,
  we may proceed as follows: first sample a random graph $\Gamma^{(k)}\sim
  G(n,p)$, then sample a second random graph $\tilde \Gamma^{(0)} \sim
  G(n,p_0/(1-p))$ \emph{independently} of $\Gamma^{(k)}$, and finally let
  $\Gamma^{(0)} = \tilde \Gamma^{(0)}\setminus \Gamma^{(k)}$. This observation
  will be important in the argument below.

  Additionally, for each possible outcome of $\Gamma^{(k)}$, we can fix
  a canonical packing $\{g_i\colon G_i-I_i\to \Gamma^{(k)}\mid i\in \mathcal B^{(k)}\}$
  that depends only on $\Gamma^{(k)}$
  (provided such a packing exists; although it does a.a.s.).

  With this in mind, we can generate the auxiliary graph $B_i$ for $i\in \mathcal B^{(k)}$ as
  follows. First expose the edges of $\Gamma^{(k)}$. By the canonical choice
  of the packing, this completely determines the embedding $g_i\colon
  G_i-I_i\to \Gamma^{(k)}$. Next, sample an independent random graph
  $\tilde \Gamma^{(0)}\sim G(n,p_0/(1-p))$ as above. Let $\tilde B_i$ be defined like $B_i$ except
  that we put an edge between $X$ and $x$ if
  $x$ is a common neighbour of $X$ in $\tilde \Gamma^{(0)}$ (instead of $\Gamma^{(0)}$).
  Finally, let $B_i$ be obtained from $\tilde B_i$ by removing the edges $(X,x)$ where
  there exists $y\in X$ such that $\{y,x\}\in E(\Gamma^{(k)})$. Since $\Gamma^{(0)}=\tilde
  \Gamma^{(0)} \setminus \Gamma^{(k)}$, this definition of $B_i$ results in the same distribution
  as the one given at the beginning of this subsection.

  Because $I_i$ is a $2$-independent set and because $\tilde\Gamma^{(0)}$ is independent
  of $\Gamma^{(k)}$ (hence independent of the sets $g_i(N_{G_i}(v))$ for $v\in S_i$),
  we see that $\tilde B_i$ is distributed as a random
  bipartite graph $B(\gamma n,p_0^{\Delta}/(1-p)^\Delta)$.
  So by Lemma~\ref{lemma:matchings-resilience}, with probability $1-o(n^{-1})$ the local resilience of
  $\tilde B_i$ with respect to containing a perfect matching is greater than
  $\gamma np_0^\Delta/2\geq 2\gamma^{1.1}n$.

  Finally, $B_i$ is obtained from $\tilde B_i$ by removing at most $\Delta\cdot
  \Delta(\Gamma^{(k)})$ edges at each vertex. However, by the Chernoff bound
  (Lemma~\ref{lemma:chernoff}) and the union bound, the maximum degree of
  $\Gamma^{(k)}$ is a.a.s.\ at most $2np \leq 2\gamma^{2} n$ for each $k\in [M]$. So even after
  deleting these edges, the local resilience of $B_i$ with respect to
  containing a perfect matchings is still greater than $2\gamma^{1.1} n -
  2\Delta\gamma^2 n \geq \gamma^{1.1} n$. The claim follows by an application
  of the union bound over all $i$.
\end{proof}

\subsection{Phase III: extending to a packing of the $G_i$}

We now complete the packing $\{g_i \colon G_i-I_i\to K_n\mid i\in [t]\}$ to a
packing $\{h_i \colon G_i\to K_n \mid i\in [t]\}$ according to the following
randomized procedure. We process the auxiliary graphs $B_1,\dotsc,B_t$ one by
one. Given $i\in [t]$ let us say that a perfect matching in $B_i$ is
\emph{eligible} if it does not contain an edge $(X,x)$ such that one of the
embeddings $h_1,\dotsc,h_{i-1}$ uses an edge $\{y,x\}$ where $y\in X$. Fix a
collection of $\gamma^{1.2}n$ edge-disjoint eligible perfect matchings in
$B_i$, assuming such a collection exists; if not, the procedure fails. Extend
$g_i$ to $h_i$ by choosing a perfect matching from this collection uniformly at
random.

If it succeeds, this procedure produces a collection of edge-disjoint
embeddings $g_i \colon G_i\to K_n$ for $i\in [t]$. We only need to show that
the algorithm succeeds with positive probability. This follows from the
following claim, which completes the proof since $t= O(n)$.

\begin{claim}
  For each $i$ the following holds: the probability that the algorithm fails in step $i$ given
  that it did not fail in any previous step is at most $o(1/n)$.
\end{claim}

\begin{proof}
  For $j\in [t]$, let $F_j := g_i(N_{G_j}(I_j)) \cup \Coim g_j$ and recall
  that $|F_j| \leq (\Delta+1)\gamma n$.

  Fix $i$ and suppose that we construct $h_1,\dotsc,h_{i-1}$ according to the
  randomized procedure described above. Say that
  an edge $\{u,v\}$ is \emph{forbidden} if it is used by one of the embeddings
  $h_1,\dotsc,h_{i-1}$. Then it is enough to show that for each $v\in F_i$
  there are at most $\gamma^{1.1}n/(2\Delta)$ vertices $u\in F_i\setminus
  \{v\}$ such that $\{u,v\}$ is an edge of $\Gamma^{(0)}$ that is forbidden. Indeed,
  this implies that at every vertex of $B_i$, there are at most
  $\gamma^{1.1}n/2$ non-eligible edges, which by the local resilience
  (Claim~\ref{claim:resilience}) implies that there is a collection of at least
  $\gamma^{1.1}n/2 \geq \gamma^{1.2} n$ edge-disjoint perfect matchings.

  Fix some vertex $v\in F_i$.
  First of all, note that an edge $\{u,v\} \in E(\Gamma^{(0)})$
  can only be used by the embedding $h_j$ if $\{u,v\}\subseteq F_j$.
  By the first statement of Claim~\ref{claim:balance}, there are at most $\gamma^{0.9} n$ values $j<i$
  such that $v\in F_j$. Additionally, for all but at most
  $\gamma^{1.2}n$ values $j < i$ such that $v\in F_j$, we have
  and $|F_i\cap F_j| \leq \gamma^{1.5} n$. Indeed, otherwise there would be
  some $u\in F_i$ such that $\{u,v\}$ is contained in $\gamma^{2.7}n^2/|F_i| \geq \gamma^{1.7} n/(\Delta+1)$
  many sets $F_j$, contradicting the second statement of Claim~\ref{claim:balance}.

  From the $\gamma^{1.2} n$ values for which this fails, we get at most
  $\Delta\gamma^{1.2} n \leq \gamma^{1.1}n/4$ forbidden edges incident to $v$.
  For the at most $\gamma^{0.9}n$ remaining ones, recall that $h_j$
  was obtained by choosing a perfect matching uniformly at random from a set of
  $\gamma^{1.2}n$ edge-disjoint perfect matchings. So the probability that
  $h_j$ uses an edge from $v$ to a vertex $u\in F_i$ is at most
  \[ \frac{|F_i\cap F_j|}{\gamma^{1.2}n} \leq \gamma^{0.3} \]
  \emph{independently} of all previous choices. Since the maximum degree of each graph $G_j$ is at most
  $\Delta$, the Azuma-Hoeffding inequality (Lemma~\ref{lemma:azuma})
  implies that with probability
  \[ 1- e^{-\frac{\gamma^{2.4}n}{2\Delta^2}} = 1-o(n^{-2}), \] there are at most
  $2\Delta \gamma^{1.2} n \leq \gamma^{1.1}n/(2\Delta)$ forbidden edges in $\Gamma^{(0)}[F_i]$ incident to $v$.
  The claim now follows by taking the union bound over all vertices $v$.
\end{proof}

\section{Proof of Lemma~\ref{lemma:bounded-gnp}}\label{sec:packing-small}

Before we can prove Lemma~\ref{lemma:bounded-gnp}, we need to collect some
auxiliary results. Recall that $K_\ell\times b$ denotes the graph consisting of
$b$ disjoint copies of $K_\ell$. Firstly, we will need the following result, which
states that there exist asymptotically optimal packings of copies
of $K_\ell\times (1-o(1))n/\ell$ in $G(n,p)$.

\begin{lemma}\label{lemma:clique-factor}
  Let $\epsilon,p\in (0,1)$ be constants and let $\ell\in \mathbb N$ be
  sufficiently large. Let $N = (1-\epsilon)\frac{np}{\ell-1}$ and let $b =
  (1-\epsilon)\frac{n}{\ell}$. Then a graph $G\sim G(n,p)$ a.a.s.\ contains a
  family of edge-disjoint subgraphs $C_1,\dotsc,C_N$ where each graph $C_i$ is
  isomorphic to $K_\ell\times b$ and where for each vertex $v$ of $G$ there are
  at least $(1-\epsilon)N$ values $s\in [N]$ such that $v\in V(C_s)$.
\end{lemma}

We note that the lemma actually holds for values of $p$ as low as
$\operatorname{polylog} n \cdot n^{-(\ell-2)/(\binom{\ell}2 -1)}$, with much
the same proof.

The proof of Lemma~\ref{lemma:clique-factor} uses a very powerful result of
Pippenger and Spencer~\cite{pippenger-spencer}. Recall that the chromatic index
$\chi'(\mathcal H)$ of a hypergraph $\mathcal H$ is defined as the smallest
number $c$ such that there exists a colouring of the edges of $\mathcal H$ with
$c$ colours in which every colour class forms a matching (i.e., a set of
independent edges). For a hypergraph $\mathcal H$, let $\delta(\mathcal H)$ be
the minimum degree of $\mathcal H$, let $\Delta(\mathcal H)$ be the maximum
degree of $\mathcal H$, and let $\Delta_2(\mathcal H)$ be the maximum co-degree
of any two distinct vertices in $\mathcal H$. Then clearly $\chi'(\mathcal H)
\geq \Delta(\mathcal H)$. Pippenger and Spencer proved that if $\delta(\mathcal
H) \sim \Delta(\mathcal H)$ and $\Delta_2(\mathcal H) = o(\Delta(\mathcal H))$,
then this lower bound is asymptotically tight:

\begin{theorem}[Pippenger and Spencer~\cite{pippenger-spencer}]
  \label{thm:ps}
  For every $k\geq 2$ and $\delta >0$, there exist $\delta'>0$ and $n_0$ such that
  if $\mathcal H$ is a $k$-uniform hypergraph on $n\geq n_0$ vertices satisfying
  \[ \delta(\mathcal H) \geq (1-\delta')\Delta(\mathcal H) \quad \text{and}\quad
  \Delta_2(\mathcal H) \leq \delta'\Delta(\mathcal H) \]
  then
  \[ \chi'(\mathcal H) \leq (1+\delta) \Delta(H). \]
\end{theorem}

We will also use the following simple fact about the distribution of cliques in
sufficiently dense random graphs. We omit the proof (it is a simple application
of the Chernoff bound).

\begin{lemma}\label{lemma:clique-regularity}
  The following holds for every integer constant $\ell \geq 3$ and every
  $p = p(n)\geq (\log n)^{-3}$.
  With probability $1-o(n^{-1})$ every edge of $G(n,p)$ is
  contained in $(1\pm o(1))\binom{n}{\ell-2}p^{\binom{\ell}{2}-1}$
  copies of $K_\ell$ and any two distinct edges of $G(n,p)$ are contained in
  $o(n^{\ell-2}p^{\binom{\ell}{2}-1})$ copies of $K_\ell$.
  Similarly, with probability $1-o(n^{-1})$, every vertex of $G(n,p)$ is contained
  in $(1\pm o(1)) \binom{n}{\ell-1}p^{\binom{\ell}{2}}$ copies of $K_\ell$.
\end{lemma}

\begin{proof}[Proof of Lemma~\ref{lemma:clique-factor}]
  Let us fix constants
  $0 < \zeta \ll \delta \ll \epsilon$
  and let $G\sim G(n,p)$. We first show that $G$ a.a.s.\ contains a large
  collection of edge-disjoint copies of $K_\ell$ such that each vertex is
  contained in at least $(1-\zeta)np/(\ell-1)$ copies.

  For this, colour each edge of $G$ with a colour chosen uniformly at random
  from $\{1,\dotsc,\log n\}$ and write $G_i$ for the graph of edges in colour
  $i$. Then $G_i \sim G(n,q)$ where $q = p/(\log n)^2 \geq (\log n)^{-3}$. By
  Lemma~\ref{lemma:chernoff} (the Chernoff bound) and the union bound, a.a.s.\
  the maximum degree of each $G_i$ is at most $(1+o(1))nq$. By
  Lemma~\ref{lemma:clique-regularity} and the union bound, a.a.s.\ each graph
  $G_i$ has the following additional properties:
  \begin{enumerate}[(i)]
    \item each vertex is contained in at least $(1-o(1))\binom{n}{\ell-1} q^{\binom{\ell}{2}}$
      copies of $K_\ell$;
    \item each edge is contained in $(1\pm o(1))\binom{n}{\ell-2} q^{\binom{\ell}{2}-1}$
      copies of $K_\ell$;
    \item any pair of distinct edges is contained in $o(n^{\ell-2} q^{\binom{\ell}{2}-1})$
      copies of $K_\ell$.
  \end{enumerate}

  For each $i$, consider the $\binom{\ell}{2}$-uniform hypergraph $\mathcal
  H_i$ whose vertices are the edges of $G_i$ and whose edges are the edge sets
  of copies of $K_\ell$ in $G_i$. Note that a matching in $\mathcal H_i$ corresponds
  to a collection of edge-disjoint copies of $K_\ell$ in $G_i$. By (ii)
  and (iii), we have \[ (1-o(1))\Delta(\mathcal H_i) \leq
  \binom{n}{\ell-2}q^{\binom{\ell}{2} -1}
  \leq (1+o(1))\delta(\mathcal H_i) \]
  and $\Delta_2(\mathcal H_i) = o(\Delta(\mathcal H_i))$ for all $i$.
  Thus by Theorem~\ref{thm:ps} the edges of each hypergraph $\mathcal H_i$ can be
  coloured with $\chi'(\mathcal H_i) = (1+o(1)) \binom{n}{\ell-2}q^{\binom{\ell}2 -1}$
  colours such that each colour class is a matching. Fix such a colouring for
  each $\mathcal H_i$.

  For a given vertex $v$ and a given colour class $M$ in the colouring of $\mathcal H_i$,
  let us write $v \cap M$ for the set of (hyper)edges in $M$ that contain an
  edge of $G_i$ incident to $v$. Since $M$ corresponds to a collection
  of edge-disjoint copies of $K_\ell$ in $G_i$, the set $v\cap M$ can also be seen as the
  set of the copies of $K_\ell$ in $M$ that contain $v$. Because $G_i$ has maximum degree at most
  $(1+o(1))nq$ we trivially have $|v\cap M| \leq (1+o(1))nq/(\ell-1)$ for any
  colour class $M$. On the other hand,
  there are fewer than $\zeta\cdot \chi'(\mathcal H_i)$ colour classes $M$ in the colouring
  of $\mathcal H_i$ such that $|v\cap M| \leq (1-\zeta)nq/(\ell-1)$: otherwise we would have
  $\sum_M |v\cap M| \leq (1-\zeta^2+o(1))\binom{n}{\ell-1}q^{\binom{\ell}{2}}$
  (where the sum is over colour classes $M$ in the colouring of $\mathcal H_i$),
  contradicting (i).

  We now construct a collection $\mathcal C_i$ of edge-disjoint copies of
  $K_\ell$ in $G_i$ as follows: choose a colour class $M_i$ in the colouring of $\mathcal H_i$
  uniformly at random and let $\mathcal C_i$ be the copies of $K_\ell$
  corresponding to the edges in $M_i$. Note that this random choice is done
  independently for each $i$. Thus, by the above considerations,  the
  probability that $|v\cap M_i|
  \leq (1-\zeta)nq/(\ell-1)$ is at most $\zeta$ -- independently of
  what happens for other $i$!
  Thus, if we let $\mathcal C := \mathcal C_1 \cup \dotsb \cup \mathcal C_{\log n}$, then
  by the Chernoff bound (Lemma~\ref{lemma:chernoff})
  and the union bound, a.a.s.\ every vertex $v$ is
  contained in at least $(1-2\zeta)np/(\ell-1)$ copies of
  $K_\ell$ in $\mathcal C$.

  Now that we have a collection $\mathcal C$ of many edge-disjoint copies of $K_\ell$, we
  use Theorem~\ref{thm:ps} a second time to find many copies of $K_\ell\times b$ in this collection.
  This time, let $\mathcal H$ be the hypergraph on $V(G)$ whose edges are the vertex sets
  of the copies of $K_\ell$ in $\mathcal C$. By the property of $\mathcal C$ obtained above,
  we have $\delta(\mathcal H) \geq (1-2\zeta)np/(\ell-1)$. On the other hand, as the maximum
  degree of $G$ is at most $(1+o(1))np$,  we have $\Delta(\mathcal H) \leq (1+o(1))np/(\ell-1)$.
  Finally, we have $\Delta_2(\mathcal H) \leq 1$. By Theorem~\ref{thm:ps}, there is a proper
  edge colouring of $\mathcal H$ with at most $(1+\delta)np/(\ell-1)$ colours. Each colour
  class corresponds to a disjoint union of copies of $K_\ell$ from $\mathcal C$.

  Since each vertex is contained in at least $(1-2\zeta)np/(\ell-1) \geq (1+\epsilon/3)N$ copies of
  $K_\ell$ in $\mathcal C$, there must be at least this many colour classes.
  The proof is completed by observing that all but at most
  $\delta np/(\ell-1) \leq \epsilon N/3$ colour classes must have size at least $(1-\delta)n/\ell \geq b$.
  Indeed, by the Handshaking Lemma, $\mathcal C$ contains at least
  $(1-2\zeta)\frac{n^2p}{\ell(\ell-1)}$
  copies of $K_\ell$, and if $\delta np/(\ell-1)$ colour classes would have size below
  $(1-\delta) n/\ell$, then there would be fewer than
  \[ \delta(1-\delta)\frac{n^2p}{\ell(\ell-1)} + (1-\delta) (1+o(1))\frac{n^2p}{\ell(\ell-1)}
  \leq (1-2\zeta)\frac{n^2p}{\ell(\ell-1)}\]
  copies. Thus there exist at least $N$ edge-disjoint copies of $K_\ell\times b$.
  Since for each vertex $v$, there are at most $(2\zeta+\delta) np/(\ell-1)
  \leq \epsilon N$ colour classes that do not contain $v$, this also gives the
  second part of the claim,
\end{proof}

Our next lemma is about packing graphs from $\mathcal{BC}((1-\epsilon)\ell,K)$ into
$K_\ell$. This is made easy for us since we can simply use Theorem~\ref{thm:mrs}. However,
more direct arguments also work; see for example~\cite{alon-yuster}.

\begin{lemma}\label{lemma:bounded-clique}
  For all constants $\epsilon\in (0,1)$ and $K>0$ there exists $\ell_0$ such
  that the following holds for all $\ell\geq \ell_0$. Any graphs
  $G_1,\dotsc,G_t\in \mathcal{BC}((1-\epsilon)\ell,K)$ such that
  $e(G_1)+\dotsb+e(G_t) \leq (1-\epsilon)\binom{\ell}{2}$
  pack into $K_\ell$.
\end{lemma}

\begin{proof}
  We can assume that $\epsilon$ is sufficiently small.
  Let $C = C(\epsilon)$ be a sufficiently large constant.
  Note that if $t\leq (1-\epsilon)\ell$, then the lemma follows immediately from
  Theorem~\ref{thm:mrs}.

  The idea is to transform the collection $G_1,\dotsc,G_t$
  into a smaller collection $\tilde G_1,\dotsc,\tilde G_s$
  with the following properties: (1) at most four graphs $\tilde G_i$ have
  fewer than $(1-\epsilon)3\ell/4$ edges; (2) each graph $\tilde G_i$ is an
  edge-disjoint union of one or more unique graphs $G_j$ (here \emph{unique} means
  that the assignment of
  graphs $G_j$ to graphs $\tilde G_i$ is a function);
  (3) $\tilde G_i$ has at most $(1-\epsilon/2) \ell$ vertices;
  (4) the maximum size of a component in  $\tilde G_i$ is at most $C K$.

  Suppose that we can do this. Then by (2) we have
  \[ e(\tilde G_1)+\dotsb + e(\tilde G_s)  = e(G_1)+\dotsb +e(G_t) \leq (1-\epsilon)\binom{\ell}{2}\]
  Together with (1) this implies easily that $s\leq (1-\epsilon) \ell$. Now by (3--4)
  and Theorem~\ref{thm:mrs}, the graphs $\tilde G_1,\dotsc,\tilde G_s$
  pack into $K_\ell$ (assuming that $\ell$ is large enough). By (2) this corresponds to a
  packing of the graphs $G_1,\dotsc,G_t$ and we are done. It remains to show
  how to obtain $\tilde G_1,\dotsc,\tilde G_s$.

  For this we use the following procedure, which operates in two phases. Let
  $\tilde G_1,\dotsc,\tilde G_s$ denote the current collection, where initially
  $s=t$ and $\tilde G_i = G_i$. By adding isolated vertices, assume also that
  $v(\tilde G_i) = (1-\epsilon) \ell$ for each $i$.

  In the first phase, we repeat the following as long as there are
  at least two graphs $\tilde G_i$ and $\tilde G_j$ with at most
  $(1-\epsilon)\ell/4$ edges. Given two such graphs, we note that they both have at least
  $(1-\epsilon)\ell/2$ isolated vertices. Therefore, by taking their disjoint union
  and removing isolated vertices, we obtain a graph on at most $(1-\epsilon)\ell$
  vertices. This new graph still has maximum component size at most $K$.

  We end up with a collection $\tilde G_1,\dotsc,\tilde G_s$ in which at most
  one graph has at most $(1-\epsilon) \ell/4$ edges. In the second phase, as
  long as there are at least three graphs $\tilde G_i, \tilde G_j,\tilde G_k$
  that have strictly more than $(1-\epsilon) \ell/4$ and at most
  $(1-\epsilon)3\ell/4$ edges, we replace these three graphs by a single graph
  as follows.

  Let $r:= (1-\epsilon/2)\ell/(CK)$. If $C$ is sufficiently large,
  then we can greedily group the components of $\tilde
  G_i$ (all of which have size bounded by $K$) into graphs
  $H_{i1},\dotsc,H_{ir}$ where each graph contains at most
  $(1-\epsilon/4)CK$ vertices. Similarly, we obtain graphs
  $H_{j1},\dotsc,H_{jr}$ and $H_{k1}, \dotsc,H_{kr}$ for $\tilde G_j$ and $\tilde G_k$.
  If $C$ is large enough, then for each $m\in [r]$ there exists a graph $H_m$
  on at most $CK$ vertices that is the edge-disjoint union of $H_{im},H_{jm},H_{km}$ (for this
  one can for example use Theorem~\ref{thm:mrs}, although much simpler arguments also work).
  Then we replace $\tilde G_i, \tilde G_j,\tilde G_k$ by the disjoint union
  $H := H_1+\dotsc+H_r$. Note that $H$ has strictly more than $(1-\epsilon)3\ell/4$ edges, so
  it will not be touched again; moreover, $v(H)\leq (1-\epsilon/2) \ell$ and every
  component in $H$ has size at most $CK$.
  This completes the proof of the lemma.
\end{proof}

We can now prove Lemma~\ref{lemma:bounded-gnp}.

\newtheorem*{replemma}{Lemma~\ref{lemma:bounded-gnp}}

\begin{replemma}
  Let $\epsilon,\alpha,\beta,K>0$ and $p\in (0,1]$ be constants and let $G\sim G(n,p)$.
  Assume that $G_1,\dotsc,G_t$ are graphs in $\mathcal{BC}((1-\epsilon)n,K)$ such that
  \[ e(G_1)+\dotsb+e(G_t)\leq (1-\epsilon)\binom{n}{2}p,\]
  where each graph $G_i$ has at least $(1-2\epsilon)n$ vertices. Moreover,
  assume that $A_1,\dotsc,A_t$ and $B_1,\dotsc,B_t$ are sets where for all $i$
  we have $A_i,B_i\subseteq V(G_i)$, $|A_i|\leq \alpha n$, and $|B_i|\leq \beta
  n$. Then a.a.s.\ there exists a packing $\{f_i\colon G_i\to G\}$ of the graphs
  $G_1,\dotsc,G_t$ into $G$ with the following properties:
  \begin{enumerate}[(1)]
    \item for every vertex $v$ of $G$ there are at most $2\alpha t$ values $i\in [t]$ such that
      $v\in f_i(A_i)$ and at most $4(\beta+\epsilon) t$ values $i\in [t]$ such that
      $v\in f_i(B_i) \cup \Coim f_i$;
    \item for any two distinct vertices $u,v$ of $G$ there are at most $5(\beta + \epsilon)^2t$ values
      $i\in [t]$ such that $\{u,v\}\subseteq f_i(B_i)\cup \Coim f_i$.
  \end{enumerate}
\end{replemma}

\begin{proof}
  Let $\ell$ be a sufficiently large constant and let $\delta>0$ be a sufficiently small constant.
  Let $N = (1-\delta)\frac{np}{\ell-1}$ and $b = (1-\delta)\frac{n}{\ell}$.
  By Lemma~\ref{lemma:clique-factor}, $G$ a.a.s.\ contains a collection of $N$
  edge-disjoint copies of $K_\ell\times b$ such that each vertex is contained
  in at least $(1-\delta)N$ copies. Let us denote by
  \[ F_s = F_{s1} + \dotsb + F_{sb} \subseteq G \]
  the $s$-th copy of $K_\ell\times b$ in this collection,
  where each graph $F_{sj}$ is a copy of $K_\ell$.

  By applying Lemma~\ref{lemma:averaging} to the sequence of vectors
  $(1,\frac{e(G_i)}{K n})\in [0,1]^2$, there exists a partition
  \[ \{1,\dotsc,t\} = \mathcal B_1 \cup \dotsb \cup\mathcal B_N \]
  which satisfies $|\mathcal B_s| = t/N\pm 6$ and
  \begin{equation}\label{eq:edges2}
    \sum_{i\in \mathcal B_s} e(G_i) = \frac{e(G_1) + \dotsb + e(G_t)}{N} \pm 6K n
    \leq (1-\epsilon/2)\frac{n(\ell-1)}2
  \end{equation}
  for all $s$, where the upper bound holds if $\ell$ is sufficiently large.

  The idea is to pack all the graphs $G_i$ for $i\in \mathcal B_s$ into the graph
  $F_s$. Since $F_s$ is isomorphic to $K_\ell\times b$, we want to
  write each such $G_i$ as a disjoint union of subgraphs of graphs $H_{i1},\dotsc,H_{ib}$.
  With this in mind, fix some $i\in [t]$ and let $\mathcal C_i$
  be the set of components of $G_i$. With
  each component $C\in \mathcal C_i$ we associate a vector
  \[ v_C = \Big(\frac{v(C)}{K}, \frac{e(C)}{K^2}, \frac{|V(C)\cap A_i|}{K},
  \frac{|V(C)\cap B_i|}{K}\Big)\in [0,1]^4. \]
  By Lemma~\ref{lemma:averaging} there exists a partition of $\mathcal C_i$ into $b$ sets
  $\mathcal C_{i1},\dotsc, \mathcal C_{ib}$ such that for each $j\in [b]$ we have
  \[ \sum_{C\in \mathcal C_{ij}} v_C \in \frac{1}{b}\Big(\frac{v(G_i)}{K},
  \frac{e(G_i)}{K^2}, \frac{|A_i|}{K},
  \frac{|B_i|}{K}\Big) + 12[-1,1]^4. \]
  Note that by letting $H_{ij}$ be the subgraph containing exatly the components $C\in \mathcal C_{ij}$,
  this partition corresponds to a way of writing $G_i$ as a disjoint union
  \begin{equation}\label{eq:decomp}
    G_i = H_{i1} + \dotsb + H_{ib}
  \end{equation}
  of graphs satisfying the following properties, where the upper bound in each
  case follows from the definition of $b$, the fact that $\ell$ is a
  sufficiently large constant, and from the bounds
  $|A_i| \leq \alpha n$ and $|B_i| \leq \beta n$:
  \begin{enumerate}[(i)]
    \item $v(H_{ij}) = \frac{v(G)}{b} \pm 12K \in \ell - [5/2, 1/2]\cdot \epsilon \ell$;
    \item $e(H_{ij}) = \frac{e(G_i)}{b} \pm 12K^2 \leq (1+\epsilon/2)e(G_i)\ell/n$;
    \item $|V(H_{ij})\cap A_i| = \frac{|A_i|}{b} \pm 12K
      \leq 1.5\alpha \ell$;
    \item $|V(H_{ij})\cap B_i| = \frac{|B_i|}{b} \pm 12K
      \leq 1.5\beta \ell$.
  \end{enumerate}

  Using Lemma~\ref{lemma:bounded-clique} it is now easy to see that for each $s$ and $j$,
  the family $\{H_{ij} \mid i \in \mathcal B_s\}$ packs into $F_{sj}$ (which is
  isomorphic to $K_\ell$): the condition on the number of vertices is (i), the
  condition on the number of edges follows from~\eqref{eq:edges2} and (ii),
  and each component of $H_{ij}$ has size at most $K$.

  Therefore we can fix edge-disjoint embeddings $\{ f_{ij} \colon H_{ij}\to
  F_{sj}\mid i\in \mathcal B_s \}$. For each $s$ and $j$, let $\sigma_{sj}$ be
  an automorphism of $F_{sj}$ chosen independently and uniformly at random
  among the $\ell!$ choices (which simply amounts to a random permutation of
  the vertices of $F_{sj}$). Clearly if for $i\in \mathcal B_s$ we set $g_{ij}
  := \sigma_{sj}\circ f_{ij}$ then each set $\{ g_{ij} \colon H_{ij}\to
  F_{sj}\mid i\in \mathcal B_s \}$ is also a packing. For each $i$, let
  $g_i\colon G_i \to G$ be the embedding obtained by combining the partial
  embeddings $g_{i1},\dotsc,g_{ib}$. Then the collection $\{g_i\mid i\in [t]\}$
  is a packing of the graphs $G_i$ into $G$.

  Since each vertex is contained in at least $(1-\delta)N$ graphs $F_s$, and since
  the embeddings in each graph $F_{sj}$ are random, it is easy to check (using the Azuma-Hoeffding
  and union bounds) that (1) and (2) are a.a.s.\ satisfied for all vertices.
\end{proof}

\section*{Acknowledgements}

We would like to thank Rajko Nenadov for useful discussions in the early stages of this project.

\end{document}